\documentclass[reqno,twoside]{amsart} 

\usepackage{amsmath}

\theoremstyle{plain}
\begingroup
\newtheorem{thm}{Theorem}[section]
\newtheorem{lem}[thm]{Lemma}
\newtheorem{prop}[thm]{Proposition}

\endgroup

\theoremstyle{definition}
\begingroup

\newtheorem{rem}[thm]{Remark}
\endgroup

\numberwithin{equation}{section}

\newcommand{\res}{\mathop{\hbox{\vrule height 7pt width .5pt depth 0pt
\vrule height .5pt width 6pt depth 0pt}}\nolimits}

\newcommand{\N}{\mathbb N} 
\newcommand{\R}{\mathbb R} 
\newcommand{\Sn}{\mathbb S} 
\newcommand{\Rn}{{\mathbb R}^n} 
\newcommand{\Ms}{{\mathbb M}^{n{\times}n}_{sym}}

\newcommand{\wto}{\rightharpoonup}
\newcommand{\e}{\varepsilon}

\newcommand{\LL}{{\mathcal L}}
\newcommand{\HH}{{\mathcal H}}
\newcommand{\M}{{\mathcal M}}
\newcommand{\C}{{\mathcal C}}

\let\O=\Omega

\begin{document}
 
\title[Traces of functions of bounded deformation]{Traces of functions of bounded deformation}
\author[J.-F. Babadjian]{Jean-Fran\c cois Babadjian}

\address[J.-F. Babadjian]{Universit\'e Pierre et Marie Curie -- Paris 6, CNRS, UMR 7598 Laboratoire Jacques-Louis Lions, Paris, F-75005, France}
\email{jean-francois.babadjian@upmc.fr}

\date{\today}

%28A33   Spaces of measures, convergence of measures
%46E35   	Sobolev spaces and other spaces of "smooth'' functions, embedding theorems, trace theorems
%26B30   	Absolutely continuous functions, functions of bounded variation


\begin{abstract}
This paper is devoted to give a simplified proof of the trace theorem for functions of bounded deformation defined on bounded Lipschitz domains of $\mathbb{R}^n$. As a consequence, the existence of one-sided Lebesgue limits on countably $\mathcal{H}^{n-1}$-rectifiable sets is also established.
\end{abstract}
\maketitle

\section{Introduction}

Functions of bounded deformation have been introduced in connexion with variational problems of small strain elasto-plasticity. These models lead to the minimization of energy functionals with linear growth with respect to the linearized strain, and, therefore, enforce finite energy configurations for which the linearized strain is an integrable function. 
Contrary to the case of elasticity, no Korn's inequality is available in the non-reflexible space $L^1$ (see \cite{O,CFM}), and usual compactness methods (adding a viscosity, Galerkin or finite difference approximations) forces one to fall within the framework of linear strains in the space of measures. 

The space $BD(\O)$ of {\it functions of bounded deformation} in an open set $\O$ of $\Rn$ has been introduced in \cite{MSC,Suquet1} as the natural energy space to study plasticity models (see \cite{Suquet3,ST2,AG}). It is made of all integrable functions $u:\O \to \Rn$ whose distributional strain -- the symmetric part of the distributional derivative $Eu:=(Du+Du^T)/2$ -- is a Radon measure with finite total variation. In the footsteps of these works, a systematic study of this space has been carried out, starting from basic properties (traces, imbeddings, Poincar\'e-Korn inequalities) in  \cite{AG,ST1,Suquet2}, to fine properties (structure of the strain measure, one-dimensional sections, approximate differentiability) in \cite{ACDM,H}, by analogy with the space $BV(\O)$ (see \cite{AFP} and references therein). We also refer to the monograph \cite{Temam} and the unpublished thesis \cite{Kohn} for additional information.

Later on, a variational model of fracture mechanics \cite{FM} gave another application of the space $BD(\O)$. This model consists in a constrained minimization of a Mumford-Shah type energy functional for which the subspace $SBD(\O)$ of {\it special functions of bounded deformation} introduced in \cite{ACDM,BCDM} (see also \cite{C}) has found relevant. Recently, the space $GBD(\O)$ of {\it generalized functions of bounded deformation} has been introduced and studied in \cite{DM} in order to palliate the lack of maximum principle in such vectorial variational problems. We also refer to \cite{BFT,E,ET,GZ,R} and references therein for applications of functions of bounded deformation to lower semicontinuity, relaxation, and related problems of the calculus of variations.

\medskip

A common typical feature of these models is that they develop strain concentration -- the so-called slip surfaces in plasticity and cracks in brittle fracture -- leading, from a macroscopic  point of view, to displacement discontinuities. For this reason, it is relevant to understand the nature of discontinuities of functions of bounded deformation which was one of the main purposes of \cite{ACDM,BCDM,C}. In particular, as in the case of functions of bounded variation, one needs to make precise the notion of trace on sets of co-dimension one. 

A first result in that direction has been given  in \cite{Suquet0} where a very weak notion of trace of $BD(\O)$ functions has been obtained as an element of the distribution space $W^{-1,\frac{n}{n-1}}(\partial \O;\Rn)$. Subsequently, this result has been improved in \cite{ST1,Suquet2} by establishing that the trace actually belongs to $L^1(\partial \O;\R^n)$. 
However, it turns out that the proofs presented in those references, only given for domains of class $\C^1$, contain several imprecisions (see the paragraph below). In addition, the arguments do not  easily extend to Lipschitz domains, although such a result has been stated and used many times. Since we did not locate a precise non-ambiguous proof in the literature, and according to the importance of this result, we decided to present in this work a rigorous proof of the trace theorem in $BD(\O)$ for Lipschitz domains $\O$.

\medskip

In order to highlight the imprecisions in \cite{ST1,Suquet2}, let us briefly sketch the general idea of the proofs in the simplified two-dimensional case ($n=2$). As usual in trace theory, a localization argument allows one to reduce to the case where $\O$ is written as the sub-graph $\{ (x_1,x_2) : x_2<a(x_1)\}$ of some $\C^1$ function $a: \RÊ\to \R$ in some orthogonal basis $(e_1,e_2)$ of $\R^2$, and $u$ has compact support in $\overline \O$. Since smooth functions are not dense in $BD(\O)$ for the norm topology, the trace cannot be defined as the unique continuous extension of the restriction mapping to the boundary. It is rather obtained by taking the limit of $u(x)$ as the point $x \in \O$ tends to the boundary. Contrary to the case of functions of bounded variation, since we only control the symmetric part of the gradient, it is not possible to control simultaneously all the components of the trace. Indeed, if $x \in \partial \O$, taking the limit of $u(x-t \xi)$ as $t \to 0^+$ with respect to some direction $\xi \in \R^2$, only allows one to get an estimate of the trace of the component of $u$ along the direction $\xi$. A formal justification of this fact is given by the following computation: for any vectors $\xi$ and $\zeta \in \R^2$ pointing outside $\O$, and any point $x \in \partial \O$,
$$u(x)\cdot \zeta =-\int_0^{+\infty} \frac{d}{dt} [u(x-t \xi) \cdot \zeta] \, dt=\int_0^{+\infty} Du(x-t\xi): (\zeta \otimes \xi)\, dt$$
which can be estimated in terms of $Eu$ only if $\zeta=\xi$ since 
$$Du: (\xi \otimes \xi)=Eu: (\xi \otimes \xi).$$ 
If $\xi$ is a transversal direction to $\partial \O$, it is proved in \cite{ST1} that the function $x \in \partial \O \mapsto u(x-t \xi) \cdot \xi$ admits a limit $g_\xi$ in $L^1(\partial \O;\Rn)$ as $t \to 0^+$ satisfying the integration by parts formula: for all $\varphi \in \C^1(\R^2)$,
\begin{equation}\label{IPP}
\int_\O (u \cdot \xi) (\nabla \varphi \cdot \xi)  \, dx + \int_\O \varphi \, d[Eu:(\xi \otimes \xi)] = \int_{\partial \O} \varphi g_\xi (\nu \cdot \xi) \, d\HH^1,
\end{equation}
where $\nu$ is the outer unit normal to $\partial \O$. Since $e_2$ is clearly transversal to $\partial \O$, it permits to define the second component $\gamma_2(u)$ of the trace of $u$ in the basis $(e_1,e_2)$ by $g_{e_2}$, and it also shows the validity of the integration by parts formula involving the second diagonal term $E_{22}u$. The definition of the first component of the trace is more involved. 

Let us first explain the strategy of \cite{ST1}. Using the implicit function theorem, whose application strongly rests on the $\C^1$ character of the boundary, it is shown that for $\lambda$ large enough, the vectors $\lambda e_2 \pm e_1$ are transversal to $\partial \O$. At this point, the authors of \cite{ST1} claim that it is possible to define the first component $\gamma_1(u)$ of the trace of $u$ in the basis $(e_1,e_2)$ by adding-up the integration by parts formulas \eqref{IPP} with $\xi=\lambda e_2+e_1$ and $\xi=\lambda e_2-e_1$, which would also establish the integration by parts formula involving the first diagonal term $E_{11}u$. Next subtracting \eqref{IPP} with $\xi=\lambda e_2+e_1$ and $\xi=\lambda e_2-e_1$ would lead to the last integration by parts formula involving the anti-diagonal term $E_{12}u$. 
Unfortunately, a careful inspection of that proof shows that this argument can be rigorously justified only if we {\it a priori} knew that $\xi \mapsto g_\xi$ is linear. Of course this property is related to the linearity of the trace mapping $\gamma$. The point is that, exactly as in the case of functions of bounded variation, the linearity (and the uniqueness) of $\gamma$ is ensured thanks the {\it full} integration by parts formula. Unfortunately, this property is not clear at this step of the proof, and this point seems to have been neglected in \cite{ST1}.

On the other hand, the argument developed in \cite{Suquet2} rests on the fact that, using again the implicit function theorem, and thus the $\C^1$ character of the boundary, there is no loss of generality to assume that the function $a:\R \to \R$ (defining the boundary curve) is one to one. Therefore, writing $\partial \O$ as $\{(x_1,x_2) : x_1=a^{-1}(x_2)\}$ permits to define correctly the first component $\gamma_1(u)$ of the trace in the basis $(e_1,e_2)$, and as a by-product to establish the integration by parts formula involving the term $E_{11} u$.  The main drawback of this approach is that the components of $\gamma(u)$ are not obtained by the same limit procedure, and therefore the last formula, involving the term $E_{12}u$, does not follow straightforwardly from the construction.

In conclusion, both arguments given in \cite{ST1,Suquet2} seem to present some ambiguities, and this is the reason why we decided to give a simplified rigorous and self contained proof of that result. Let us close this introduction by pointing out the recent result of \cite{DM} where a notion of trace for a generalized class of functions of bounded deformation,  called $GBD$, has been introduced as the approximate limit when points tend to the boundary (see Section 5 of that paper). In particular, the existence of traces on submanifolds of class $\C^1$ and on Lipschitz boundaries has been established by means of a method similar to that used in \cite{ST1}. If restricting to $BD$ fields, it would be possible to infer the stronger trace results, Theorem \ref{thm:BD} and Proposition \ref{prop:rect}, of the present paper at the expense of the integration by parts formula which was actually the missing point of \cite{ST1}. We however militate in favor of our approach since it only uses elementary tools of standards $BD$ functions, in contrast with \cite{DM} where the analysis of generalized $BD$ functions needs a more complicated treatment and finer geometric measure theoretic arguments.

\medskip

The paper is organized as follows. After gathering the main notation used throughout this work in Section \ref{sec:notations}, we give in Section \ref{sec:trace} a self contained proof of the trace theorem in $BD(\O)$ for bounded open sets $\O \subset \Rn$ with Lipschitz boundary, on the one hand generalizing the results of \cite{ST1,Suquet2}, and, on the other hand, getting rid of the imprecisions present in both papers. As a by-product, we also show in Section \ref{sec:rect} the existence of one-sided Lebesgue limits on countably $\HH^{n-1}$-rectifiable sets. The proofs presented below follow the lines of the analogous result for functions of bounded variation in \cite{EG}. 

\section{Notation}\label{sec:notations}

\subsection{Vectors and matrices}

If $x$ and $y \in \Rn$ we use the notation $x \cdot y$ for the scalar product in $\Rn$, and $|x|$ for the norm. We denote by $\mathbb M^{n \times n}$ the set of all real $n \times n$ matrices, and by $\Ms$ the subset of all  symmetric matrices. If $A$ and $B \in \mathbb M^{n \times n}$, we write $A:B$ for the scalar product in $\mathbb M^{n \times n}$, and $|A|$ for the norm. We recall that for any two vectors $a$ and $b \in \R^n$, $a \otimes b \in \mathbb M^{n \times n}$ stands for the tensor product, {\it i.e.}, $(a \otimes b)_{ij}=a_i b_j$ for all $1 \leq i ,j\leq n$, and $a \odot b:=(a \otimes b + b \otimes a)/2 \in \Ms$ is the symmetric tensor product. It is easy to check that 
\begin{equation}\label{eq:ab}
|a \odot b|\geq \frac{1}{\sqrt 2}|a| |b|.
\end{equation}

\subsection{Sets}

If $x \in \Rn$ and $\varrho>0$, then $B_\varrho(x)$ stands for the open ball of center $x$ and radius $\varrho$. If $x=0$, we simply write $B_\varrho$ instead of $B_\varrho(0)$. The Lebesgue measure in $\R^n$ is denoted by $\LL^n$, and the $(n-1)$-dimensional Hausdorff measure by $\HH^{n-1}$. We use the notation $\Sn^{n-1}:=\{Ê\zeta \in \R^n : |\zeta|=1\}$ for the unit sphere of $\Rn$ and $\omega_n=\LL^n(B_1)$, so that $\HH^{n-1}(\Sn^{n-1})=n \omega_n$. 

Given a vector $\xi \in \Sn^{n-1}$, we write $\Pi_\xi:=\{x \in \Rn : x \cdot \xi=0\}$ for the hyperplane orthogonal to $\xi$ and passing through the origin. We denote by $\pi_\xi$ the orthogonal projection onto $\Pi_\xi$, {\it i.e.}, $\pi_\xi(x)=x-(x\cdot \xi) \xi \in \Pi_\xi$ for any $x \in \R^n$.

We say that a set $\Gamma \subset \Rn$ is countably $\HH^{n-1}$-rectifiable if $\Gamma= \bigcup_{i=1}^\infty \Gamma_i \cup N$, where $\Gamma_i \subset M_i$ for some $(n-1)$-dimensional  manifolds $M_i \subset \R^n$ of class $\C^1$, while $N \subset \Rn$ is a $\HH^{n-1}$-negligible set (see \cite[Section 2.9]{AFP}). It is possible to define, for $\HH^{n-1}$-a.e. $x \in \Gamma$, an approximate unit normal, denoted by $\nu_\Gamma(x)\in \Sn^{n-1}$, and characterized by
\begin{equation}\label{approx}
\nu_\Gamma(x)=\pm \nu_{M_i}(x) \quad \text{for all } i \geq 1, \quad\text{and for }\HH^{n-1}\text{-a.e. }x \in \Gamma \cap M_i,
\end{equation}
where $\nu_{M_i}(x)$ is a normal to $M_i$ at the point $x$ (see \cite[Section 2.11]{AFP}).

\subsection{Measures}

Given an open subset $\O$ of $\R^n$ and a finite dimensional Euclidian space $X$, we denote by $\M(\O;X)$ the space of all $X$-valued Radon measures with finite total variation. If $X=\R$, we simply write $\M(\O;\R)=\M(\O)$. According to the Riesz representation Theorem, $\M(\O;X)$ can be identified to the topological dual of $\C_0(\O;X)$ (the space of all continuous functions $\varphi : \O \to X$ such that $\{|\varphi| \geq \e\}$ is compact for every $\e>0$), and a weak* topology is defined according to this duality. If $\mu \in \M(\O;X)$ then $|\mu|$ stands for the variation measure, and if $A \subset \O$ is a Borel set, the measure $\mu \res A$ is defined by $(\mu \res A)(B)=\mu (A \cap B)$ for every Borel set $B \subset \O$.

Another natural weak* topology in the space of measures is that generated by the topological dual of $\C_b(\O;X)$ (the space of bounded and continuous functions from $\O$ to $X$). We recall, in particular, that if $(\mu_k) \subset \M(\O;X)$ and $\mu \in \M(\O;X)$, then (see {\it e.g.}Ê Proposition 1.80 and Theorem 2.39 in \cite{AFP}),
\begin{equation}\label{eq:mu}
\begin{cases}
\mu_k \wto \mu \text{ weakly* in }\M(\O;X),\\
|\mu_k|(\O) \to |\mu|(\O),
\end{cases}
\Longrightarrow
\begin{cases}
|\mu_k| \wto |\mu| \text{ weakly* in }[\C_b(\O)]',\\
\mu_k \wto \mu\text{ weakly* in }[\C_b(\O;X)]'.
\end{cases}
\end{equation}

\subsection{Functions of bounded deformation}

The space of functions of {\it bounded deformation}, denoted by $BD(\O)$, is the space of all functions $u \in L^1(\O;\R^n)$ whose symmetric part of the distributional derivative satisfies 
$$Eu:=\frac{Du+Du^T}{2} \in \M(\O;\Ms).$$
The space $BD(\O)$ is a Banach space when endowed with the norm 
$$\|u\|_{BD(\O)}:=\|u\|_{L^1(\O)} + |Eu|(\O).$$ 
It is proved in \cite[Proposition 2.5]{ST1} that $BD(\O)$ can be identified to the dual of a Banach space, and therefore it can be endowed with a natural weak* topology. It turns out that a sequence $(u_k) \subset BD(\O)$ converges weakly* in $BD(\O)$ to some $u \in BD(\O)$ if and only if $u_k \to u$ strongly in $L^1(\O;\Rn)$ and $Eu_k \wto Eu$ weakly* in $\M(\O;\Ms)$. An intermediate notion of convergence between weak* and strong convergences is the so-called {\it strict convergence}: a sequence $(u_k) \subset BD(\O)$ converges strictly to some $u \in BD(\O)$ if and only if 
$$\begin{cases}
u_k \to u \text{ strongly in }L^1(\O;\R^n),\\
Eu_k \wto Eu \text{ weakly* in }\M(\O;\Ms),\\
|E u_k|(\O) \to |E u|(\O).
\end{cases}$$

Let us finally recall several results about functions of bounded deformation that will be used throughout this work. Let us stress that, of course, all these results are independent of the trace theorem. It is known that smooth functions are not dense in $BD(\O)$. However, weaker approximation results hold: for any $u \in BD(\O)$ there exists a sequence $(u_k) \subset \C^\infty(\O;\R^n) \cap BD(\O)$  such that $u_k \to u$ strictly (see \cite[Theorem 1.3]{AG}). If, in addition, $\O$ is bounded with Lipschitz boundary, the same result holds with $\C^\infty(\overline \O;\R^n)$ in place of $\C^\infty(\O;\R^n) \cap BD(\O)$ (see \cite[Theorem II-3.2]{Temam}). In the sequel, we shall also use the fact that, if $u \in BD(\O)$, the measure $Eu$ does not charge $\HH^{n-1}$ negligible sets (see \cite[Remark 3.3]{ACDM}).

\section{Trace on the boundary of Lipschitz domains}\label{sec:trace}

The object of this section is to show a trace theorem in $BD(\O)$ for bounded open sets $\O \subset \Rn$ with Lipschitz boundary, generalizing the results of  \cite[Theorem 1.1]{ST1} and \cite[Theorem 1]{Suquet2} for $\C^1$ domains. Note that the proofs of \cite{ST1,Suquet2} were strongly using the $\C^1$ character of the boundary through different applications of the implicit function theorem. We employ here another approach similar to that used in the case of functions of bounded variation \cite{EG}. It consists in using the density of smooth functions with respect to the strict convergence.

As explained in the introduction, the main difference with the case of functions of bounded variation is that, since we only control the symmetric part of the gradient of $u \in BD(\O)$, it is not possible to directly estimate at the same time all components of the trace with respect to the $BD(\O)$-norm of $u$. Indeed, if $\partial \O$ is locally written as the graph of some Lipschitz function in some direction, say, $e_n \in \Sn^{n-1}$, then one can only define the component of the trace with respect to this direction $e_n$. In order to define the $n-1$ other components, one should be able to slightly move the local coordinate system by still keeping the graph property of $\partial \O$. This is the object of the following well-known  geometric result.

\begin{lem}\label{lem:geom}
Let $(e_1,\ldots,e_n)$ be an orthonormal basis of $\R^n$, and let $a : \R^{n-1} \to \R$ be a Lipschitz function in this coordinate system (where $\R^{n-1}$ is identified to ${\rm Vect}(e_1,\ldots,e_{n-1})$, and $\R$ is identified to ${\rm Vect}(e_n)$). We define the Lipschitz graph  of $a$ by
$$\Sigma:=\{x=(x',x_n)\in \R^n : x_n = a(x')\},$$
where $x'=(x_1,\ldots,x_{n-1})$. 
Then there exists $\eta_0>0$ (depending only on the Lipschitz constant of $a$) such that for any $\xi \in \Sn^{n-1}$ with $|\xi - e_n| < \eta_0$, we have
$$\Sigma=\{x \in \R^n : x \cdot \xi = a_\xi(x - (x\cdot \xi) \xi)\},$$
for some Lipschitz function $a_\xi :\Pi_\xi \to \R$.
\end{lem}

\begin{proof}
Let $L>0$ be the Lipschitz constant of $a$. Let us fix
\begin{equation}\label{alpha}
\alpha := \frac{L}{\sqrt{ 1+L^2}},
\end{equation}
and define the open cone
$$C:=\{Ê\zeta \in \R^n : |\zeta \cdot e_n| > \alpha |\zeta|\}.$$

\medskip

{\bf Step 1.} Let us show that, for any $x \in \Sigma$, $(x+\overline C) \cap \Sigma =\{x\}$. For any $\xi \in \overline C$, and for all $x \in \Sigma$, we define the straight line passing through $x$ and with direction $\xi$ by
$$L_x^\xi:=\{x+t\xi,\; t \in \R\}.$$
We first observe that $L_x^\xi\cap \Sigma \neq \emptyset $ since $x \in L_x^\xi \cap \Sigma$. Next, if $y \in L_x^\xi \cap \Sigma$ is any other point, then we can write it as $y=x+t \xi$ for some $t \in \R$, and thus, since $y \in \Sigma$, then $a(y')=y_n=y \cdot e_n = (x+t\xi) \cdot e_n=x_n+t\xi \cdot e_n$. Moreover, since $x \in \Sigma$, we also have that $x_n=a(x')$, whence
\begin{equation}\label{eq:a}
|a(y')-a(x')|=|y_n-x_n| = |t| |\xi \cdot e_n| > \alpha |tÊ\xi|  = \alpha |x-y|.
\end{equation}
On the other hand, Pythagoras Theorem ensures that $|x-y|^2 = |x_n-y_n|^2 + |x'-y'|^2 = |a(x')-a(y')|^2 + |x'-y'|^2$. Thus gathering with \eqref{eq:a} leads to $|a(y') - a(x')|^2 > \alpha^2|a(y') - a(x')|^2 + \alpha^2|x'-y'|^2$, or still by \eqref{alpha},
$$|a(y')-a(x')|^2> L^2 |x'-y'|^2.$$
But since the function $a$ is $L$-Lipschitz, we deduce that $y'=x'$, and thus $y=x$.

\medskip

{\bf Step 2.} Let us prove that for every $\xi \in C \cap \Sn^{n-1}$, there exists $\beta>0$ such that 
$$C_\xi:=\{\zeta \in \R^n : |\zeta \cdot \xi |> \beta |\zeta|\}Ê\subset C.$$
Note that since $C_\xi$ and $C$ are cones, it is enough to check that $C_\xi \cap \Sn^{n-1} \subset C$. Let us define $c_0:=(|\xi\cdot e_n|-\alpha) \wedge \sqrt{2} >0$, $\beta:=Ê(2-c_0^2)/2>0$, and let $\zeta \in C_\xi \cap \Sn^{n-1}$. Assuming first $\xi \cdot \zeta \geq 0$ leads to  $|\zeta - \xi|^2 = 2 - 2\zeta \cdot \xi < 2-2\beta=c_0^2$. Therefore, $|\zeta \cdot e_n| = |\xi \cdot e_n + (\zeta-\xi) \cdot e_n| \geq  |\xi \cdot e_n| - |\zeta -\xi|>  |\xi \cdot e_n|-c_0 \geq \alpha$. If instead $\xi \cdot \zeta \leq 0$, then  $|\zeta + \xi|^2 = 2 +2\zeta \cdot \xi < 2-2\beta=c_0^2$, and  again $|\zeta \cdot e_n| = |-\xi \cdot e_n + (\zeta+\xi) \cdot e_n| \geq  |\xi \cdot e_n| - |\zeta +\xi|>  |\xi \cdot e_n|-c_0 \geq \alpha$.

\medskip

{\bf Step 3.} Let us show the conclusion of the lemma. From Steps 1 and 2, for any $\xi \in C \cap \Sn^{n-1}$ and any $x \in \Sigma$ we have $(x+\overline{C_\xi}) \cap \Sigma=\{x\}$, and thus 
\begin{equation}\label{eq:sexy}
\Sigma \subset x+{}^c C_\xi \quad \text{ for every } \xi \in C \text{ and } x \in \Sigma.
\end{equation}
But if $\zeta \in {}^c C_\xi$, then $|\zeta \cdot \xi| \leq \beta |\zeta|$, and Pythagoras Theorem ensures that 
$$\frac{1}{\beta^2} |\zeta \cdot \xi|^2 \leq |\zeta|^2 = |\zeta \cdot \xi|^2 + |\pi_\xi(\zeta)|^2,$$
or still
\begin{equation}\label{eq:beta}
|\zeta \cdot \xi|^2 \leq \frac{\beta^2}{1-\beta^2} |\pi_\xi(\zeta)|^2\quad \text{ for every }\zeta \in {}^c C_\xi.
\end{equation}

Let us define the mapping $a_\xi : \pi_\xi(\Sigma) \to \R$ by $a_\xi(y):=t$, where $y+tÊ\xi \in \Sigma$. Let us check that this map is well defined. Indeed if $y_1$ and $y_2 \in \pi_\xi(\Sigma)$, then there exist
$x_1$ and $x_2 \in \Sigma$ such that $y_i=\pi_\xi(x_i)$ for $i=1$, $2$. Then \eqref{eq:sexy} shows that $x_1 \subset x_2 +{}^c C_\xi$, and thus $x_1-x_2 \in {}^c C_\xi$. Hence \eqref{eq:beta} yields in turn
\begin{equation}\label{eq:axi}
|x_1 \cdot \xi- x_2 \cdot \xi|^2 \leq \frac{\beta^2}{1-\beta^2} |\pi_\xi(x_1)-\pi_\xi(x_2)|^2.
\end{equation}
Therefore, since $a_\xi(y_i)=x_i \cdot \xi$ for $i=1$, $2$, then
$$y_1=y_2 \Rightarrow \pi_\xi(x_1)=\pi_\xi(x_2) \Rightarrow x_1 \cdot \xi= x_2 \cdot \xi \Rightarrow a_\xi(y_1)=a_\xi(y_2).$$
In addition, \eqref{eq:axi} also shows that $a_\xi$ is Lipschitz since
$$|a_\xi(y_1)- a_\xi(y_2)|^2 \leq \frac{\beta^2}{1-\beta^2} |y_1 - y_2|^2.$$
It is thus possible to extend $a_\xi$ as a Lipschitz function from $\Pi_\xi$ to $\R$ with the same Lipschitz constant. The conclusion is obtained by defining $\eta_0:=\sqrt{2-2\alpha}$, and noticing that if $\xi \in \Sn^{n-1}$ is such that $|\xi-e_n| <\eta_0$, then $\xi \in C$.
\end{proof}

As for $BV$ functions, smooth functions up to the boundary are not dense in $BD$ for the norm topology. Therefore the trace mapping cannot be obtained as the continuous extension of the restriction mapping to the boundary, as is usually done in Sobolev spaces. It will follow from the approximation of $BD$ functions by smooth ones with respect to the strict convergence. However, this construction does not {\it a priori} ensure the uniqueness of such a mapping nor its linearity. Nevertheless, it turns out that both properties actually hold among all maps satisfying the integration by parts formula, from which the trace in $BD$ is thus indissociable. The proof presented below is inspired from that in the $BV$ case (see \cite[Theorem 5.3.1]{EG}).

\begin{thm}\label{thm:BD}
Let $\O \subset \Rn$ be a bounded open set with Lipschitz boundary. There exists a unique linear continuous mapping
$$\gamma : BD(\O) \to L^1(\partial \O;\R^n)$$
such that the following integration by parts formula holds: for every $u \in BD(\O)$ and $\varphi \in \C^1(\R^n)$,
\begin{equation}\label{eq:IPP2}
\int_\O u \odot \nabla \varphi\,  dx +\int_\O \varphi \, dEu = \int_{\partial \O} \gamma(u) \odot \nu\,  \varphi \, d\HH^{n-1},
\end{equation}
where $\nu$ is the  outer unit normal to $\partial \O$. In addition, 
\begin{equation}\label{eq:C0}
\gamma(u)=u|_{\partial \O} \text{ for all }u \in \C^0(\overline \O;\R^n) \cap BD(\O).
\end{equation}
\end{thm}

\begin{proof}

{\bf Step 1.} We first assume that $u \in \C^\infty(\overline \O;\Rn)$. In this case, we simply define the trace of $u$ on $\partial \O$ by the restriction of $u$ on that set: $\gamma(u):=u|_{\partial \O}$. This mapping is clearly linear, and we next prove that it is continuous with values in $L^1(\partial \O;\R^n)$.

\medskip

{\bf Step 1a.} In this step we study locally the trace mapping defined above. Let $x_0 \in \partial \O$, then there exists an open set $A' \subset \R^n$ containing $x_0$, an orthonormal basis $(e_1,\ldots,e_n)$ of $\R^n$, and (in this coordinate system) a Lipschitz mapping $a : \R^{n-1} \to \R$ such that
$$\begin{cases}
\O \cap A'= \{x=(x',x_n) \in A' : x_n <a(x')\},\\
\partial \O \cap A' = \{x=(x',x_n) \in A' : x_n =a(x')\}.
\end{cases}$$
We will prove that  for any open set $A \subset \subset A'$,
\begin{equation}\label{eq:cont-trace-LD0}
\int_{\partial \O \cap A} |\gamma(u)|\, d\HH^{n-1} \leq C \|u\|_{BD(\O \cap A')},
\end{equation}
where the constant $C>0$ is independent of $u$.

Let us define the graph of $a$ by $\Sigma:=\{x=(x',x_n) \in \Rn : x_n =a(x')\}$. According to Lemma \ref{lem:geom}, there exists $\eta_0>0$ such that for any $\xi \in \Sn^{n-1}$ with $|\xi - e_n| < \eta_0$, then $\Sigma=\{x \in \R^n : x \cdot \xi = a_\xi(x - (x\cdot \xi) \xi)\}$ for some Lipschitz function $a_\xi :\Pi_\xi \to \R$. Let $\xi \in \Sn^{n-1}$ with $|\xi - e_n| < \eta_0$, and let $\e_0>0$ be such that, for all $\e \leq \e_0$, the open sets 
$$A_\e^\xi:=\{z=y-t\xi : y \in \partial \O \cap A, \; 0<t<\e\}$$
are contained in $\O \cap A'$.
%$$\O_\e^\xi:=\{x \in A : x \cdot \xi < a_\xi(x-(x\cdot \xi) \xi) -\e\}$$
%is not empty, and  for all $y \in \partial \O \cap A$, the points $y-\e \xi$ belong to $\O \cap A'$. 
For all $y \in \partial \O \cap A$ and all $\e \in (0,\e_0]$, by the fundamental theorem of calculus, we have
$$u(y-\e\xi) \cdot \xi - u(y) \cdot \xi= \int_0^\e \frac{d}{dt} [u(y-t\xi) \cdot \xi] \, dt= -\int_0^\e Eu(y-t\xi)\xi \cdot \xi \, dt,$$
and thus, Fubini's Theorem yields
$$\int_{\partial \O\cap A} |u(y-\e\xi) \cdot \xi - u(y) \cdot \xi|\, d\HH^{n-1}(y) \leq \int_0^\e\int_{\partial \O \cap A} |Eu(y-t\xi)\xi \cdot \xi|\, d\HH^{n-1}(y) \, dt.$$
%for every open set $B \subset A$. Next, the coarea formula applied to the Lipschitz function $x \mapsto x \cdot \xi - a_\xi\big(x-(x\cdot \xi) \xi\big)$ leads to
Using the area formula for Lipschitz graphs together with Fubini's Theorem, we get that 
\begin{eqnarray}\label{eq:estimBD}
\int_{\partial \O \cap A} |u(y-\e\xi) \cdot \xi - u(y) \cdot \xi|\, d\HH^{n-1}(y) & \leq & \int_{A_\e^\xi} |Eu\xi \cdot \xi| \sqrt{1+|\nabla a_\xi|^2} \, dx\nonumber\\
& \leq & C_\xi \int_{A_\e^\xi} |Eu\xi \cdot \xi| \, dx,
\end{eqnarray}
where $C_\xi>0$ only depends on the Lipschitz constant of $a_\xi$, and, in particular,
$$\int_{\partial \O \cap A} |\gamma(u) \cdot \xi|\, d\HH^{n-1} \leq C_\xi \int_{\O \cap A'} |Eu| \, dx + \int_{\partial \O \cap A} |u(y-\e \xi)|\, d\HH^{n-1}(y).$$
Integrating the previous inequality with respect to $\e \in (0,\e_0]$, and invoking once more Fubini's Theorem leads to
\begin{multline*}
\int_{\partial \O \cap A} |\gamma(u) \cdot \xi|\, d\HH^{n-1} \leq C_\xi \int_{\O \cap A'} |Eu| \, dx +\frac{1}{\e_0}\int_0^{\e_0} \int_{\partial \O \cap A} |u(y-t\xi)|\, d\HH^{n-1}(y)\, dt\\
\leq C_\xi\int_{\O \cap A'} |Eu| \, dx +\frac{1}{\e_0}\int_{A^\xi_{\e_0}} |u|\sqrt{1+|\nabla a_\xi|^2} \, dx\\
\leq C_\xi \int_{\O \cap A'} |Eu| \, dx +\frac{C_\xi}{\e_0}  \int_{\O\cap A'} |u| \, dx \leq C_\xi \|u\|_{BD(\O \cap A')}.
\end{multline*}

Taking first $\xi=e_n$ leads to $\gamma_n(u):=\gamma(u) \cdot e_n$ which satisfies the estimate
\begin{equation}\label{eq:gamman}
\int_{\partial \O \cap A} |\gamma_n(u)|\, d\HH^{n-1} \leq C_n \|u\|_{BD(\O \cap A')}.
\end{equation}
For $i\in \{1,\ldots,n-1\}$, let us take $\xi=\xi_i:=\frac{e_n+\delta e_i}{\sqrt{1+\delta^2}}$, where $\delta>0$ is small enough so that $|\xi_i-e_n|<\eta_0$. Then
$$\int_{\partial \O \cap A} |\gamma(u)\cdot \xi_i|\, d\HH^{n-1} \leq C_i \|u\|_{BD(\O \cap A')},$$
and, defining the other components of $\gamma(u)$ in the basis $(e_1,\ldots,e_n)$ by 
$$\gamma_i(u):=\gamma(u)\cdot e_i=\frac{\sqrt{1+\delta^2} \; \gamma(u)\cdot \xi_i - \gamma_n(u)}{\delta}$$ 
leads to
\begin{equation}\label{eq:gammai}
\int_{\partial \O \cap A} |\gamma_i(u)|\, d\HH^{n-1} \leq \left( C_i \frac{\sqrt{1+\delta^2}}{\delta}+\frac{C_n}{\delta} \right) \|u\|_{BD(\O \cap A')}.
\end{equation}
Combining \eqref{eq:gamman} and \eqref{eq:gammai} gives \eqref{eq:cont-trace-LD0}.

\medskip

{\bf Step 1b.} We now extend the local estimate \eqref{eq:cont-trace-LD0} into a global estimate. Since $\partial \O$ is compact, it can be covered by finitely many open sets $A_1,\ldots,A_N$ with the following properties: for all $i \in \{1,\ldots,N\}$, there exist an orthonormal basis $(e_1,\ldots,e_n)$ of $\R^n$ (depending on $i$), a Lipschitz function $a _i: \R^{n-1} \to \R$ (in this coordinate system), and open sets $A'_i$ with $A_i \subset\subset A'_i$ such that
$$\begin{cases}
\O \cap A'_i= \{x=(x',x_n) \in A'_i : x_n <a_i(x')\},\\
\partial \O \cap A'_i = \{x=(x',x_n) \in A'_i : x_n =a_i(x')\}.
\end{cases}$$
Let $\theta_1, \ldots,\theta_N$ be a partition of unity subordinated to this covering, {\it i.e.}, $\theta_i \in \C_c^\infty(A_i;[0,1])$, and $\sum_{i=1}^N \theta_i=1$ on $\partial \O$. Since $\theta_i u \in \C^\infty(\overline{\O};\R^n)$, then its trace $\gamma(\theta_i u)$ on $\partial \O \cap A_i$ satisfies \eqref{eq:cont-trace-LD0}, {\it i.e.},
$$\int_{\partial \O \cap A_i} |\gamma(\theta_i u)|\, d\HH^{n-1} \leq C \|\theta_i u\|_{BD(\O \cap A'_i)}\leq  C \| u\|_{BD(\O)},$$
where the constant $C>0$ is independent of $u$. In addition, since $\gamma(\theta_i u)=\theta_i|_{\partial \O} \gamma(u)$ and ${\rm Supp}\, (\theta_i u) \subset A_i \cap \overline \O$, we infer that 
\begin{equation}\label{eq:cont-trace-LD}
\int_{\partial \O} |\gamma(u)|\, d\HH^{n-1}\leq \sum_{i=1}^N \int_{\partial \O \cap A_i}\theta_i |\gamma( u)|\, d\HH^{n-1}\leq  C \| u\|_{BD(\O)}.
\end{equation}

\medskip

{\bf Step 2.} Let us now consider a function $u \in BD(\O)$. Then there exists a sequence $(u_k) \subset \C^\infty(\overline \O;\R^n)$ such that
\begin{equation}\label{eq:Euk}
\begin{cases}
u_k \to u \text{ strongly in }L^1(\O;\R^n),\\
Eu_k \wto Eu \text{ weakly* in }\M(\O;\Ms),\\
|E u_k|(\O) \to |E u|(\O).
\end{cases}
\end{equation}
According to Step 1, the trace $\gamma(u_k)$ of $u_k$ is well defined, and according to \eqref{eq:cont-trace-LD},
\begin{equation}\label{eq:cont-trace-LD2}
\int_{\partial \O} |\gamma(u_k)|\, d\HH^{n-1} \leq C \|u_k\|_{BD(\O)},
\end{equation}
for some constant $C>0$ independent of $k$. 

\medskip

{\bf Step 2a.} We first argue locally as in Step 1a with the same notation. We show that the sequence $(\gamma(u_k))_{k \in \N}$ is a Cauchy sequence in $L^1(\partial \O \cap A;\Rn)$.
Let us fix $0<t<\e<\e_0$. 
According to \eqref{eq:estimBD} and since $A_t^\xi \subset A_\e^\xi$, we infer that
$$\int_{\partial \O \cap A} |u_k(y-t\xi)\cdot \xi - \gamma(u_k)(y)\cdot \xi|\, d\HH^{n-1}(y) \leq C_\xi |Eu_k|(A_t^\xi) \leq C_\xi |Eu_k|(A_\e^\xi).$$
As a consequence, for any integers $k$, $lÊ\in \N$, we infer that
\begin{multline*}
\int_{\partial \O \cap A}Ê|\gamma(u_k)\cdot \xi - \gamma(u_l)\cdot \xi|\, d\HH^{n-1}\\
\leq \int_{\partial \O \cap A} |u_k(y-t\xi)\cdot \xi - \gamma(u_k)(y)\cdot \xi|\, d\HH^{n-1}(y)\\
+\int_{\partial \O \cap A} |u_l(y-t\xi)\cdot \xi- \gamma(u_l)(y)\cdot \xi|\, d\HH^{n-1}(y)\\
+\int_{\partial \O \cap A}|u_k(y-t\xi)\cdot \xi-u_l(y-t\xi)\cdot \xi |\, d\HH^{n-1}(y)\\
\leq C_\xi \bigg(|Eu_k|(A_\e^\xi)+ |Eu_l|(A_\e^\xi)\\
 +\int_{\partial \O \cap A}|u_k(y-t\xi)-u_l(y-t\xi) |\, d\HH^{n-1}(y) \bigg).
\end{multline*}
Integrating the previous inequality with respect to $t \in (0,\e)$, and using Fubini's Theorem in the last integral, we get that 
\begin{multline*}
\int_{\partial \O \cap A}Ê|\gamma(u_k)\cdot \xi - \gamma(u_l)\cdot \xi|\, d\HH^{n-1}\\
\leq C_\xi \left(|Eu_k|(A_\e^\xi)+ |Eu_l|(A_\e^\xi) +\frac{1}{\e}\int_{ \O \cap A'}|u_k-u_l|\, dx \right).
\end{multline*}

As in Step 1a, we first choose $\xi=e_n$, and then, for $i\in \{1,\ldots,n-1\}$, $\xi=\xi_i:=\frac{e_n+\delta e_i}{\sqrt{1+\delta^2}}$ for some $\delta>0$ small enough. Remembering that 
$$\gamma_n(u_k):=\gamma(u_k) \cdot e_n,\; \gamma_i(u_k)=\frac{\sqrt{1+\delta^2}\gamma(u_k) \cdot \xi_i -\gamma_n(u_k)}{\delta}\; \forall\, i\in \{1,\ldots,n-1\},$$
we obtain the estimate
\begin{multline*}
\int_{\partial \O \cap A}Ê|\gamma(u_k) - \gamma(u_l)|\, d\HH^{n-1}
\leq C_{\xi,\delta} \Bigg[ \frac{1}{\e} \int_{\O\cap A'} |u_k-u_l|\, dx  +|Eu_k|(A_\e^{e_n})+ |Eu_l|(A_\e^{e_n}) \\
+\sum_{i=1}^{n-1} \left( |Eu_k|\big(A_\e^{\xi_i}\big)+ |Eu_l|\big(A_\e^{\xi_i}\big)\right) \Bigg].
\end{multline*}

Let us observe that the convergences \eqref{eq:Euk} and \eqref{eq:mu} ensure that $|Eu_k|Ê\wto |Eu|$ weakly* in $[\C_b(\O)]'$. Passing to the upper limit as $k$ and $l \to \infty$ in the previous inequality, we get that
\begin{multline*}
\limsup_{k,l\to \infty} \int_{\partial \O \cap A}Ê|\gamma(u_k) - \gamma(u_l)|\, d\HH^{n-1}\\
\leq C_{\xi,\delta} \Bigg[ |Eu|\big(\overline{A_\e^{e_n}}\big)+ |Eu|\big(\overline{A_\e^{e_n}}\big) +\sum_{i=1}^{n-1} \left( |Eu|\big(\overline{A_\e^{\xi_i}}\big)+ |Eu|\big(\overline{A_\e^{\xi_i}}\big)\right) \Bigg].
\end{multline*}
Since the sets $\overline{A_\e^{e_n}}$ and $\overline{A_\e^{\xi_i}}$ monotonically decrease to $A \cap \partial \O$, while $|Eu|$ in concentrated in $\O$, we deduce that the right hand side of the previous inequality tends to zero as $\e \searrow 0$. Therefore the sequence $(\gamma(u_k))_{k \in \N}$ is of Cauchy type in $L^1(\partial \O \cap A;\Rn)$, and thus, we can find some function $\gamma_A(u) \in L^1(\partial \O \cap A;\Rn)$ such that $\gamma(u_k) \to \gamma_A(u)$ strongly $L^1(\partial \O \cap A;\Rn)$. 
In addition, thanks to the usual integration by parts formula, we have for all $\varphi \in \C^1_c(A)$
$$\int_{\O \cap A} u_k \odot \nabla \varphi  \, dx +\int_{\O\cap A} \varphi Eu_k\, dx = \int_{\partial \O \cap A}Ê\gamma(u_k) \odot \nu\, \varphi \, d\HH^{n-1}.$$
At this point, we remark that the convergences \eqref{eq:Euk} and \eqref{eq:mu} ensure that $Eu_k \wto Eu$ weakly* in $[\C_b(\O;\Ms)]'$. It is then possible to pass to the limit in the previous formula, and get that
$$\int_{\O \cap A} u \odot \nabla \varphi  \, dx +\int_{\O\cap A} \varphi \, d Eu = \int_{\partial \O \cap A}Ê\gamma_A(u) \odot \nu\, \varphi \, d\HH^{n-1}.$$

\medskip

{\bf Step 2b.} We now define the trace of $u$ on the whole boundary. Using the same notation as in Step 1b, we cover $\partial \O$ by finitely many open sets $A_1,\ldots,A_N$. Step 2a ensures the existence, for each $i \in \{1,\ldots,N\}$, of a function $\gamma_{A_i}(u) \in L^1(\partial \O \cap A_i;\Rn)$ such that
$$\int_{\O \cap A_i} u \odot \nabla \varphi \, dx +\int_{\O\cap A_i} \varphi \, dEu = \int_{\partial \O \cap A_i}Ê\gamma_{A_i}(u) \odot \nu\, \varphi \, d\HH^{n-1}$$
for every $\varphi \in \C^1_c(A_i)$. In particular, taking $\varphi \in \C^1_c(A_i \cap A_j)$, for $i \neq j$, yields
$$\int_{\partial \O \cap A_i\cap A_j}Ê\gamma_{A_i}(u) \odot \nu\, \varphi \, d\HH^{n-1}=\int_{\partial \O \cap A_i\cap A_j}Ê\gamma_{A_j}(u) \odot \nu\, \varphi \, d\HH^{n-1}$$
which implies that $\gamma_{A_i}(u) \odot \nu=\gamma_{A_j}(u) \odot \nu$ $\HH^{n-1}$-a.e. on $\partial \O \cap A_i\cap A_j$. According to \eqref{eq:ab}, we infer that $\gamma_{A_i}(u)=\gamma_{A_j}(u)$ $\HH^{n-1}$-a.e. on $\partial \O \cap A_i\cap A_j$. It is thus possible to define the mapping $\gamma : BD(\O) \to L^1(\partial \O;\Rn)$
by setting $\gamma(u):=\gamma_{A_i}(u)$ $\HH^{n-1}$-a.e. on $\partial \O \cap A_i$. It has the property that $\gamma(u_k) \to \gamma(u)$ strongly in $L^1(\partial \O \cap A_i;\Rn)$ for all $i \in \{1,\ldots,N\}$, and consequently,
\begin{equation}\label{eq:conv-trace}
\gamma(u_k) \to \gamma(u) \quad \text{strongly in }L^1(\partial \O;\Rn).
\end{equation}
In particular, passing to the limit in \eqref{eq:cont-trace-LD2}, and invoking the convergences \eqref{eq:Euk} yields
\begin{equation}\label{eq:cont-trace-BD2}
\int_{\partial \O} |\gamma(u)|\, d\HH^{n-1} \leq C \|u\|_{BD(\O)}.
\end{equation}

\medskip

{\bf Step 3.} Let us show the integration by parts formula \eqref{eq:IPP2}, and, as a byproduct, the uniqueness and the linearity of $\gamma$. Since $u_k \in \C^\infty(\overline \O;\R^n)$, the usual integration by parts formula implies that for all $\varphi \in \C^1(\R^n)$,
$$\int_\O u_kÊ\odot \varphi \, dx + \int_\O \varphi  Eu_k\, dx= \int_{\partial \O} \gamma(u_k) \odot \nu \, \varphi \, d\HH^{n-1}.$$
According to the convergences \eqref{eq:Euk}, \eqref{eq:mu} and \eqref{eq:conv-trace}, we are in position to pass to the limit in the previous equality to get that 
$$\int_\O uÊ\odot \varphi \, dx + \int_\O \varphi  \, dEu= \int_{\partial \O} \gamma(u) \odot \nu \, \varphi \, d\HH^{n-1}.$$

It remains to show that the trace mapping $\gamma$ is unique, linear and continuous. To show the uniqueness, assume that $\tilde \gamma:BD(\O) \to L^1(\partial \O;\R^n)$ is another map satisfying the integration by parts formula \eqref{eq:IPP2}. Then for each $\varphi \in \C^1(\R^n)$, we infer that 
$$ \int_{\partial \O} \gamma(u) \odot \nu \, \varphi \, d\HH^{n-1}= \int_{\partial \O} \tilde \gamma(u) \odot \nu \, \varphi \, d\HH^{n-1},$$
which, thank to \eqref{eq:ab}, leads to $\gamma(u) =\tilde \gamma(u)$ $\HH^{n-1}$-a.e. on $\partial \O$. The same argument can be reproduced to show that $\gamma$ is linear, and the continuity is therefore a consequence of \eqref{eq:cont-trace-BD2}.

\medskip

{\bf Step 4.} We finally prove \eqref{eq:C0}. Using the same notation as that of Step 1a, we first show that for any $u \in BD(\O)$,
\begin{equation}\label{eq:estimBD2}
\int_0^\e \int_{\partial \O \cap A} |u(y-t\xi)\cdot \xi - \gamma(u)(y) \cdot \xi|\, d\HH^{n-1}(y)\, dt \leq C_\xi \e  |Eu|\big(\overline{A_\e^\xi}\big).
\end{equation}
Indeed, let $(u_k) \subset \C^\infty(\overline \O;\Rn)$ be a sequence as in \eqref{eq:Euk}. Then according to \eqref{eq:estimBD}, we have
\begin{multline}\label{eq:estimBD1}
\int_0^\e \int_{\partial \O \cap A} |u_k(y-t\xi)\cdot \xi  - \gamma(u_k)(y) \cdot \xi|\, d\HH^{n-1}(y)\, dt\\
\leq C_\xi \int_0^\e \int_{A_t^\xi} |Eu_k\xi \cdot \xi| \, dx\, dt
\leq  C_\xi \e |Eu_k|(A_\e^\xi),
\end{multline}
since $A_t^\xi \subset A_\e^\xi$ whenever $t<\e$. We first note that, according to Fubini's Theorem
\begin{multline*}
\int_0^\e \int_{\partial \O \cap A} |u(y-t\xi) \cdot \xi - u_k(y-t\xi) \cdot \xi|\, d\HH^{n-1}(y)\, dt \\
= \int_{A_\e^\xi} |u_k - u|Ê\sqrt{1+|\nabla a_\xi|^2}\, dx
\leq  C_\xi \int_{\O \cap A'}Ê|u_k-u|\, dx \to 0,
\end{multline*}
while, by \eqref{eq:mu}, we have $|Eu_k|\wto |Eu|$ weakly* in $[\C_b(\O)]'$. Relation \eqref{eq:estimBD2} then follows from \eqref{eq:conv-trace} by taking the limit as $k \to \infty$ in \eqref{eq:estimBD1}. 

Taking in particular $u \in BD(\O) \cap \C(\overline \O;\Rn)$, then \eqref{eq:estimBD2} leads to 
\begin{equation}\label{eq:C01}
\int_{\partial \O \cap A} \left|\frac{1}{\e}\int_0^\e u(y-t\xi)\cdot \xi \, dt - \gamma(u)(y) \cdot \xi\right|\, d\HH^{n-1}(y)\leq C_\xi |Eu|\big(\overline{A_\e^\xi}\big).
\end{equation}
On the other hand, the continuity of $u$ ensures that for all $y \in \partial \O \cap A$,
$$\frac{1}{\e} \int_0^\e u(y-t\xi)\cdot \xi\, dt \to u(y)\cdot \xi$$
as $\e \to 0$. Passing to the limit in \eqref{eq:C01} yields $u(y)\cdot \xi=\gamma(u)(y) \cdot \xi$ for all $y \in \partial \O \cap A$, and \eqref{eq:C0} follows from the fact that the previous property holds for every $\xi \in \Sn^{n-1}$ with $|\xi -e_n|<\eta_0$ which contains a basis of $\Rn$.
\end{proof}

\begin{rem}\label{rmk:cont-trace}
In the proof of Theorem \ref{thm:BD}, we actually proved that if $u \in BD(\O)$ and $(u_k) \subset \C^\infty(\overline \O;\Rn)$ is such that $u_k \to u$ strictly in $BD(\O)$, then $\gamma(u_k) \to \gamma(u)$ strongly in $L^1(\partial \O;\Rn)$. 
\end{rem}

By construction, the trace mapping $BD(\O) \to L^1(\partial \O;\Rn)$ is continuous with respect to the norm topology of $BD(\O)$, nevertheless, it is not weakly* continuous. The following result states that the trace is continuous with respect to the strict convergence of $BD(\O)$ (see \cite{ST1,Suquet2}).

\begin{prop}\label{prop:cont-trace}
Let $\O\subset \Rn$ be a bounded open set with Lipschitz boundary. Let $u \in BD(\O)$ and $(u_k) \subset BD(\O)$ be such that $u_k \to u$ strictly in $BD(\O)$. Then $\gamma(u_k) \to \gamma(u)$ strongly in $L^1(\partial \O;\Rn)$.
\end{prop}

\begin{proof}
According to \cite[Theorem II-3.2]{Temam}, for each $k \in \N$, there exists a sequence $(u_k^j) \subset \C^\infty(\overline \O;\Rn)$ such that $u^j_k \to u_k$ strictly in $BD(\O)$ as $j \to \infty$, and Remark \ref{rmk:cont-trace} ensures that $\gamma(u_k^j) \to \gamma(u_k)$ strongly in $L^1(\partial \O;\Rn)$ as $j \to \infty$. It is thus possible to find an increasing sequence $(j_k) \nearrow \infty$ such that for each $k\geq 1$,
$$\|u_k^{j_k}Ê- u_k\|_{L^1(\O)}+ \big| |Eu_k^{j_k}|(\O) - |Eu_k|(\O)\big|+\|\gamma(u_k^{j_k})Ê- \gamma(u_k)\|_{L^1(\partial \O)}<\frac{1}{k}.$$
We have thus constructed a sequence $(u_k^{j_k}) \subset \C^\infty(\overline \O;\Rn)$ such that $u_k^{j_k} \to u$ strictly in $BD(\O)$, which also satisfies $\gamma(u_k^{j_k}) \to \gamma(u)$ strongly in $L^1(\partial \O;\Rn)$ thanks again to Remark \ref{rmk:cont-trace}. We finally deduce that $\gamma(u_k) \to \gamma(u)$ strongly in $L^1(\partial \O;\Rn)$.
\end{proof}

In the following result, we show that the pointwise values of the trace $\gamma(u)$ of a function $u \in BD(\O)$ can be recovered by taking limits of averages of $u$ on balls centered on the boundary, and intersected with $\O$. 

\begin{prop}\label{prop:trace-pc}
Let $\O\subset \Rn$ be a bounded open set with Lipschitz boundary, and $u \in BD(\O)$. Then for $\HH^{n-1}$-a.e. $x \in \partial \O$,
$$\lim_{\varrho \to 0} \frac{1}{\varrho^n}\int_{B_\varrho(x) \cap \O} |u(y) - \gamma(u)(x)|\, dy =0.$$
\end{prop}

\begin{proof}
Let $x \in \partial \O$ be a Lebesgue point of $\gamma(u)$ which also satisfies
$$\lim_{\varrho \to 0} \frac{|Eu|(B_\varrho(x) \cap \O)}{\varrho^{n-1}}=0.$$
Note that $\HH^{n-1}$ almost every points of $\partial \O$ satisfy these properties. Indeed, the first one is a consequence of Lebesgue's differentiation theorem. For what concerns the second property, it suffices to show that, for each $k\geq 1$, the Borel sets
$$B_k:=\left\{x \in \partial \O : \limsup_{\varrho \to 0} \frac{|Eu|(B_\varrho(x) \cap \O)}{\omega_{n-1}\varrho^{n-1}}\geq \frac{1}{k} \right\}$$
are $\HH^{n-1}$ negligible. This property is a consequence \cite[Theorem 2.56]{AFP} since we have $\HH^{n-1}(B_k) \leq k |Eu|(B_k) =0$, because the measure $Eu$ is concentrated on $\O$.

By the Lipschitz regularity of $\partial \O$, there exist an open neighborhood  $A \subset \R^n$ of $x$, an orthonormal basis $(e_1,\ldots,e_n)$ of $\R^n$, and a Lipschitz mapping $a : \R^{n-1} \to \R$ such that
$$\begin{cases}
\O \cap A= \{x=(x',x_n) \in A : x_n <a(x')\},\\
\partial \O \cap A = \{x=(x',x_n) \in A : x_n =a(x')\}.
\end{cases}$$
As before, we define the graph of $a$ by $\Sigma:=\{x=(x',x_n) \in \Rn : x_n =a(x')\}$, and according to Lemma \ref{lem:geom}, there exists $\eta_0>0$ such that for any $\xi \in \Sn^{n-1}$ with $|\xi - e_n| < \eta_0$, then $\Sigma=\{x \in \R^n : x \cdot \xi = a_\xi(x - (x\cdot \xi) \xi)\}$ for some Lipschitz function $a_\xi :\Pi_\xi \to \R$ with Lipschitz constant $L>0$. It is enough to check that for any $\xi \in \Sn^{n-1}$ with $|\xi-e_n|<\eta_0$, then
$$\lim_{\varrho \to 0} \frac{1}{\varrho^n}\int_{B_\varrho(x) \cap \O} |u(y)\cdot \xi - \gamma(u)(x)\cdot \xi|\, dy =0,$$
since, as already seen, this family of vectors $\xi$ contains a basis of $\R^n$.

For simplicity, we denote by $c:=\sqrt{1+L^2}$. Let $\varrho>0$ be small enough so that $B_{5c\rho}(x) \subset A$, and
$$B_\varrho(x) \cap \O \subset \left\{z=y-t\xi : y \in \partial \O \cap B_{2c\varrho}(x), \, 0<t<2c \varrho\right\} \subset\subset B_{5c\rho}(x) \cap \O.$$
By this choice of $\varrho$, Fubini's Theorem implies that 
\begin{multline*}
\int_{B_\varrho(x)\cap \O}  |u(y)\cdot \xi - \gamma(u)(x)\cdot \xi|\, dy\\
 \leq \int_0^{2c\varrho} \int_{\partial \O \cap B_{2c\varrho}(x)} |u(y-t\xi)\cdot \xi - \gamma(u)(x)\cdot \xi|\, d\HH^{n-1}(y)\, dt,
\end{multline*}
and thus
\begin{multline*}
\int_{B_\varrho(x)\cap \O} |u(y)\cdot \xi - \gamma(u)(x)\cdot \xi|\, dy\\
 \leq 2 c \varrho  \int_{\partial \O \cap B_{2c\varrho}(x)} |\gamma(u)(y)\cdot \xi - \gamma(u)(x)\cdot \xi|\, d\HH^{n-1}(y)\\
+  \int_0^{2c \varrho}\int_{\partial \O \cap B_{2c \varrho}(x)} |u(y-t\xi)\cdot \xi - \gamma(u)(y)\cdot \xi|\, d\HH^{n-1}(y)\, dt.
\end{multline*}
Therefore, by \eqref{eq:estimBD2} and our choice of $x$ we infer that
\begin{multline*}
\frac{1}{\varrho^n}\int_{B_\varrho(x)\cap \O} |u(y)\cdot \xi - \gamma(u)(x)\cdot \xi|\, dy\\ 
\leq \frac{2c}{\varrho^{n-1}} \int_{\partial \O \cap B_{2c\varrho}(x)} |\gamma(u)(y) - \gamma(u)(x)|\, d\HH^{n-1}(y)\\
+ 2c\,  C_\xi \frac{ |Eu|(B_{5c\varrho}(x) \cap \O) }{\varrho^{n-1}}\to 0,
\end{multline*}
which completes the proof of the Proposition.
\end{proof}

\section{Traces on rectifiable sets}\label{sec:rect}

In this section we show that functions of bounded deformation admit one-sided Lebesgue limits on countably $\HH^{n-1}$-rectifiable subsets of $\O$. This result strongly relies on the pointwise characterization of the trace stated in Proposition \ref{prop:trace-pc}. 
\begin{prop}\label{prop:rect}
Let $\O$ be an open subset of $\Rn$, $u \in BD(\O)$, and $\Gamma$ be a countably $\HH^{n-1}$-rectifiable subset of $\O$. Then for $\HH^{n-1}$-a.e. $x \in \Gamma$, there exist one-sided Lebesgue limits $u_\Gamma^\pm(x)$ with respect to the approximate unit normal $\nu_\Gamma(x)$ to $\Gamma$, {\it i.e.},
$$\lim_{\varrho \to 0^+}\int_{B^\pm_\varrho(x,\nu_\Gamma(x))} |u(y) - u_\Gamma^\pm(x)|\, dy =0,$$
where $B^\pm_\varrho(x,\nu_\Gamma(x)):=\{y \in B_\varrho(x) : \pm (y-x) \cdot \nu_\Gamma(x) >0\}$. In addition, we have the representation
$$Eu \res \Gamma=(u_\Gamma^+ - u_\Gamma^-) \odot \nu_\Gamma \HH^{n-1} \res \Gamma.$$
\end{prop}

\begin{proof}
{\bf Step 1.} ÊWe first consider a $(n-1)$-dimensional manifold $M$ of class $\C^1$. Then for each $x_0 \in M$, there exist an open ball $U$ centered at $x_0$ such that $$M \cap U=\{ x=(x',x_n) \in U : x_n=a(x')\},$$ 
for some function $a: \R^{n-1} \to \R$ of class $\C^1$. Let us define the open sets with Lipschitz boundary
$$U^\pm:=\{x = (x',x_n) \in U :\pm x_n>\pm a(x')\}.$$
Let us denote by $\nu_M$ the unit normal to $M$ oriented from $U^-$ to $U^+$, {\it i.e.}
$$\nu_M(y)=\frac{(-\nabla a(y'),1)}{\sqrt{1+|\nabla a(y')|^2}} \quad \text{for all }y \in M \cap U,$$
so that, on $M \cap U$, the normal to $U^-$ is $\nu_M$, while the normal to $U^+$ is $-\nu_M$.
Since $u \in BD(U^\pm)$, according to Theorem \ref{thm:BD}, there exist $u_M^\pm \in L^1(M \cap U;\Rn)$ (the restriction to $M \cap U$ of the trace of $u|_{U^\pm}$) such that for all $\varphi \in \C^1_c(U)$,
$$\int_{U^\pm} u \odot \nabla \varphi\,  dx +\int_{U^\pm} \varphi \, dEu =\mp \int_{M \cap U} u^\pm_M \odot \nu_M\,  \varphi \, d\HH^{n-1}.$$
Summing up both previous relations, and using the definition of the distributional derivative, we infer that
\begin{multline*}
\int_U \varphi \, dEu = - \int_U u \odot \nabla \varphi\, dx=- \int_{U^+} u \odot \nabla \varphi\, dx-\int_{U^-} u \odot \nabla \varphi\, dx\\
=\int_{U^+}  \varphi \, dEu+\int_{U^-}  \varphi \, dEu+\int_{M \cap U} (u_M^+-u_M^-) \odot \nu_M\,  \varphi \, d\HH^{n-1},
\end{multline*}
and thus
$$\int_{M \cap U} \varphi \, dEu =\int_{M \cap U} (u_M^+-u_M^-) \odot \nu_M\,  \varphi \, d\HH^{n-1}.$$
By density, the previous relation holds for any $\varphi \in \C_0(U)$ which implies that 
\begin{equation}\label{eq:repM}
Eu \res (M\cap U)=(u_M^+ - u_M^-) \odot \nu_M \HH^{n-1} \res (M \cap U).
\end{equation}

At this point, the traces $u_M^\pm$ might depend on the open set $U$. We now show that for $\HH^{n-1}$-a.e. $x \in M \cap U$, the triplet $(u_M^+(x),u_M^-(x),\nu_M(x))$ is independent of the local representation of $M$. To this aim, we will prove that $u^\pm_M(x)$ are one-sided Lebesgue limits with respect to the direction $\nu_M(x)$. According to Proposition \ref{prop:trace-pc}, we have that  for $\HH^{n-1}$-a.e. $x \in M \cap U$,
\begin{equation}\label{eq:nuM}
\lim_{\varrho \to 0} \frac{1}{\varrho^n}\int_{B_\varrho(x) \cap U^\pm} |u(y) - u_M^\pm(x)|\, dy =0.
\end{equation}
Let us fix a point $x \in M \cap U$ ({\it i.e.} $x_n=a(x')$) satisfying \eqref{eq:nuM}. Since $a$ is of class $\C^1$, for each $\e>0$, there exists $\delta>0$ such that if $|y'-x'|<\delta$, then
$|a(y') - a(x') - \nabla a(x') (y'-x')|<\e |x'-y'|$. For all $\varrho<\delta$, defining the half balls by $B^\pm_\varrho(x,\nu_M(x)):=\{y \in B_\varrho(x) : \pm (y-x) \cdot \nu_M(x) >0\}$, then 
\begin{multline}\label{eq:last-estim}
\frac{1}{\varrho^n}\int_{B^\pm_\varrho(x,\nu_M(x))} |u(y) - u_M^\pm(x)|\, dy\\
 \leq \frac{1}{\varrho^n}\int_{B^\pm_\varrho(x,\nu_M(x)) \cap U^\pm} |u(y) - u_M^\pm(x)|\, dy\\
+ \frac{1}{\varrho^n}\int_{B^\pm_\varrho(x,\nu_M(x)) \cap U^\mp} |u(y) - u_M^\mp(x)|\, dy\\
+ |u^+_M(x)-u^-_M(x)|\frac{\LL^n(B^\pm_\varrho(x,\nu_M(x)) \cap U^\mp)}{\varrho^n}.
\end{multline}
Thanks to \eqref{eq:nuM}, the two first integrals in the right hand side of \eqref{eq:last-estim} tend to zero as $\varrho \to 0^+$. Concerning the last term, we observe that if $y \in B^\pm_\varrho(x,\nu_M(x)) \cap U^\mp$, then $|x-y|<\varrho$,
$$\pm y_n < \pm a(y'), \quad \pm \nu_M(x)\cdot (y-x) = \frac{\mp \nabla a(x')\cdot (y'-x') \pm (y_n-x_n)}{\sqrt{1+|\nabla a(x')|^2}}>0,$$
and therefore, since $|y'-x'|<\varrho<\delta$, then
$$\pm a(x') \pm \nabla a (x')\cdot (y'-x')  <\pm y_n <\pm a(x') \pm \nabla a (x')\cdot (y'-x')\pm \e|x'-y'|.$$
As a consequence, $\LL^n(B^\pm_\varrho(x,\nu_M(x)) \cap U^\mp) \leq \omega_{n-1}\e \varrho^n$, 
and by the arbitrariness of $\e$, the last term of the right hand side of \eqref{eq:last-estim} is infinitesimal as well. We have thus proved that for $\HH^{n-1}$-a.e. $x \in M \cap U$,
\begin{equation}\label{M1}
\lim_{\varrho \to 0}\frac{1}{\varrho^n}\int_{B^\pm_\varrho(x,\nu_M(x))} |u(y) - u_M^\pm(x)|\, dy=0.
\end{equation}
According to \cite[page 163]{AFP}, it shows that the triplet $(u_M^+(x),u_M^-(x),\nu_M(x))$ is uniquely defined up to a permutation of $(u_M^+(x),u_M^-(x))$ and a change of sign of $\nu_M(x)$. 

This property allows one to define, $\HH^{n-1}$-a.e. on $M$, one sided-Lebesgue limits $u_M^\pm$ with respect to a normal direction $\nu_M$. In addition, a simple covering argument permits to extend the local representation formula \eqref{eq:repM} of $Eu$ on $M$ into the global representation
\begin{equation}\label{M2}
Eu \res M=(u_M^+ - u_M^-) \odot \nu_M \HH^{n-1} \res M.
\end{equation}

{\bf Step 2.} Let us now consider a countably $\HH^{n-1}$-rectifiable set $\Gamma \subset \O$. By definition, 
$$\Gamma=\bigcup_{i=1}^\infty\Gamma_i \cup N,$$
where $\HH^{n-1}(N)=0$, and $\Gamma_i \subset M_i$ for some hypersurfaces $M_i$ of class $\C^1$. In addition, we can assume without loss of generality that $\Gamma_i \cap \Gamma_j =\emptyset$ if $i \neq j$. According to \eqref{M1}, for all $i \in \N$, there exists a set $N_i \subset \O$ with $\HH^{n-1}(N_i)=0$ with the following property: for all $x \in \Gamma_i \setminus N_i$, there exists a triplet $(u_{M_i}^+(x),u_{M_i}^-(x),\nu_{M_i}(x))$ satisfying
$$\lim_{\varrho \to 0}\frac{1}{\varrho^n}\int_{B^\pm_\varrho(x,\nu_{M_i}(x))} |u(y) - u_{M_i}^\pm(x)|\, dy=0.$$
Moreover, by the characterization \eqref{approx} of the approximate normal, for each $i \geq 1$, there exists a set  $Z_i \subset \O$ such that $\HH^{n-1}(Z_i)=0$ and $\nu_\Gamma(x) = \pm \nu_{M_i}(x)$ for all $x \in \Gamma_i \setminus Z_i$.

Let us define the exceptional set $Z:=N \cup \bigcup_{i=1}^\infty(N_i \cup Z_i)$. Then $\HH^{n-1}(Z)=0$, and we define the one sided-traces on $\Gamma \setminus Z$ by
$$u_\Gamma^\pm(x)=
\begin{cases}
u^\pm_{M_i}(x) \text{ if } x \in \Gamma_i \setminus Z \text{ and }\nu_\Gamma(x) =  \nu_{M_i}(x),\\
u^\mp_{M_i}(x) \text{ if } x \in \Gamma_i \setminus Z \text{ and } \nu_\Gamma(x) =  -\nu_{M_i}(x),
\end{cases}$$
so that, for all $x \in \Gamma \setminus Z$,
$$\lim_{\varrho \to 0}\frac{1}{\varrho^n}\int_{B^\pm_\varrho(x,\nu_\Gamma(x))} |u(y) - u_\Gamma^\pm(x)|\, dy=0.$$
In addition, since according to  \cite[Remark 3.3]{ACDM}, one has $|Eu|(Z)=0$,  \eqref{M2} yields
\begin{multline*}
Eu \res \Gamma = \sum_{i=1}^\infty Eu \res (\Gamma_i \setminus Z) =\sum_{i=1}^\infty(u_{M_i}^+ - u_{M_i}^-) \odot \nu_{M_i} \HH^{n-1} \res (\Gamma_i \setminus Z) \\
=\sum_{i=1}^\infty(u_\Gamma^+ - u_\Gamma^-) \odot \nu_\Gamma \HH^{n-1} \res (\Gamma_i \setminus Z)=(u_\Gamma^+ - u_\Gamma^-) \odot \nu_\Gamma \HH^{n-1} \res \Gamma,
\end{multline*}
and the proof is complete.
\end{proof}

\section*{Acknowledgements} 

This research has been supported by the {\sl Agence Nationale de la Recherche} under Grant No.\ ANR 10-JCJC 0106. The author wishes to thank Vincent Millot and Gilles Francfort for useful discussions concerning the content of this work.

\end{document}